\documentclass[12pt]{article}
\usepackage{amsmath,amssymb,amsthm}
\usepackage{mathabx}
\usepackage[active]{srcltx}

\usepackage[all]{xy}
\usepackage[active]{srcltx}

\newtheorem{theorem}{Theorem}[section]
\newtheorem{lemma}{Lemma}[section]

\theoremstyle{definition}
\newtheorem{remark}{Remark}[section]
\newtheorem{example}{Example}[section]
\newtheorem{definition}{Definition}[section]
\newtheorem{algorithm}{Algorithm}[section]

\renewcommand{\le}{\leqslant}
\renewcommand{\ge}{\geqslant}

\DeclareMathOperator{\rank}{rank}

\DeclareMathOperator{\diag}{diag}

\begin{document}

\title{A canonical form for
nonderogatory matrices under unitary
similarity\thanks{Linear Algebra Appl.
435 (2011) 830--841.}}

\author{Vyacheslav Futorny\thanks{Department of
Mathematics, University of S\~ao Paulo,
Brazil. Email:
\mbox{futorny@ime.usp.br}. Supported in
part by the CNPq grant (301743/2007-0)
and by the Fapesp grant
(2010/50347-9).}
   \and
Roger A. Horn\thanks{Department of
Mathematics, University of Utah, Salt
Lake City, Utah 84103, Email:
\mbox{rhorn@math.utah.edu}.}
    \and
Vladimir V.
Sergeichuk\thanks{Corresponding author.
Institute of Mathematics,
Tereshchenkivska 3, Kiev, Ukraine.
Email: \mbox{sergeich@imath.kiev.ua}.
Supported in part by the Fapesp grant
(2010/07278-6). The work was done while
this author was visiting the University
of S\~ao Paulo, whose hospitality is
gratefully acknowledged.}}

\date{}

\maketitle

\begin{abstract}
A square matrix is nonderogatory
if its Jordan blocks have distinct
eigenvalues. We give canonical forms for
\begin{itemize}
  \item nonderogatory complex
      matrices up to unitary
      similarity, and
  \item pairs of
      complex matrices up to
      similarity, in which one
      matrix has distinct
      eigenvalues.
\end{itemize}
The types of these canonical forms are given
by undirected and, respectively,
directed graphs with no undirected cycles.

{\it AMS classification:} 15A21

{\it Keywords:} Belitskii's algorithm;
Littlewood's algorithm; Unitary
similarity; Classification; Canonical
matrices
 \end{abstract}

\section{Introduction}
A square matrix is
\emph{nonderogatory} if its Jordan
blocks have distinct eigenvalues; that
is, if its characteristic and minimal
polynomials coincide.

We give canonical forms for
\begin{itemize}
  \item nonderogatory matrices up
      to unitary similarity, and
  \item pairs of
      matrices up to similarity, in
      which one matrix has
      distinct eigenvalues.
\end{itemize}
All matrices that we consider are
complex matrices.

Our canonical matrices are special
cases of the canonical matrices that
were algorithmically constructed by
Littlewood and Belitskii:
\begin{itemize}
  \item Littlewood \cite{lit}
      developed an algorithm that
      reduces each square matrix
      $M$ by
      unitary similarity transformations
\begin{equation*}\label{hec}
M\mapsto U^{-1}MU,\qquad U\text{ is a unitary matrix},
\end{equation*}
      to a matrix
      $M_{\text{can}}$ in such a
      way that $M$ and $N$ are
      unitarily similar if and only
      if they are reduced to the
      same matrix
      $M_{\text{can}}=N_{\text{can}}$.
      Thus, the matrices that are
      not changed by Littlewood's
      algorithm are \emph{canonical
      with respect to unitary
      similarity}. Other versions
     of Littlewood's algorithm were
     given in \cite{ben-cra} and
     \cite{ser_fun,ser_unit}.

  \item Belitskii \cite{bel,bel1}
      developed an algorithm that
      reduces each pair of $n\times
      n$ matrices $(M,N)$ by
      similarity transformations
\begin{equation}\label{hep}
(M,N)\mapsto (S^{-1}MS,S^{-1}NS),
\qquad S\text{ is nonsingular},
\end{equation}
      to a matrix pair
      $(M,N)_{\text{can}}$ in such
      a way that $(M,N)$ and
      $(M',N')$ are similar if and
      only if they are reduced to
      the same matrix pair
      $(M,N)_{\text{can}}=(M',N')_{\text{can}}$.
      Thus, the matrix pairs that
      are not changed by
      Belitskii's algorithm are
      \emph{canonical with respect
      to similarity}. Belitskii's
      algorithm was extended in
      \cite{ser_can} to the problem
      of classifying arbitrary
      systems of linear mappings
and the problem of classifying
representations of finite
dimensional algebras.
\end{itemize}

Lists of Littlewood's canonical
$5\times 5$ matrices and Belitskii's
canonical pairs for $4\times 4$ matrices
are in \cite{klim} and
\cite{g-s}.  Without restrictions on
the size of matrices, we cannot expect
to have explicit descriptions of Littlewood's
canonical matrices and Belitskii's
canonical matrix pairs since
\begin{itemize}
  \item The problem of classifying
      matrices up to unitary
      similarity contains the
      problem of classifying
      arbitrary systems of linear
      mappings on unitary spaces
      \cite{kr_sam,ser_unit}; and

  \item The problem of classifying
      matrix pairs up to similarity
      contains the problem of
      classifying arbitrary systems
      of linear mappings on vector
      spaces
      \cite{gel-pon,bel-ser_compl}.
\end{itemize}

When it is applied to
nonderogatory matrices, Littlewood's
algorithm can be greatly simplified.
Mitchell \cite{mit} presented an
algorithm intended to reduce nonderogatory
matrices to canonical form, but his
algorithm is incorrect\footnote{The
following reasoning on page 71 of
\cite{mit} is incorrect: ``Let us agree
to go from left to right down the
successive diagonals below the main
diagonal and pick out each non-zero
element as we come to it until we
obtain either a total of $n-1$ non-zero
elements or all non-zero elements off
the main diagonal, where $n$ is the
order of the matrix. These chosen
non-zero elements can then be made
positive by transforming by a diagonal
unitary matrix.'' Unfortunately, it is impossible to
make positive in this way ``each non-zero element
as we come to it''. In Section \ref{s3}
we choose a set of nonzero elements
that can be made positive.}. In
Sections \ref{s2}--\ref{s5} we give
a version of Littlewood's
algorithm for nonderogatory
matrices and describe a set of
canonical nonderogatory matrices for
unitary similarity. Each type of
canonical nonderogatory matrices with
$t$ distinct eigenvalues is given by an undirected
graph with $t$ vertices and no
cycles.

When it is applied to pairs of
$n\times n$ matrices in which one matrix has
distinct eigenvalues,
Belitskii's algorithm can also be greatly simplified.
 In Section
\ref{s6} we describe a
set of canonical forms for pairs of matrices under similarity.
It is analogous to the set of canonical
nonderogatory matrices in Section
\ref{s4}, but it involves directed graphs
instead of undirected graphs. This
description was used in \cite{g-s} to
classify pairs of $4\times 4$
matrices up to similarity.

\section{Schur's triangular form for
nonderogatory matrices}\label{s2}

Schur's unitary triangularization
theorem \cite[Theorem 2.3.1]{HJ1} ensures that each
square matrix $M$ is unitarily similar
to an upper triangular matrix
\begin{equation}\label{rkr}
A=\begin{bmatrix}
     \lambda _1&a_{12}&\dots&a_{1n} \\
     &\lambda _2&\ddots&\vdots \\
     &&\ddots&a_{n-1,n}\\ 0&&&\lambda _n
   \end{bmatrix},\quad \lambda_1\preccurlyeq
\lambda_2\preccurlyeq\dots\preccurlyeq
\lambda_n,
\end{equation}
whose diagonal entries are complex
numbers in any prescribed order; for
definiteness, we use the lexicographic
order:
\begin{equation}\label{gep}
\text{$a+bi\preccurlyeq c+di$\qquad if either $a<c$,
or $a=c$ and $b\le d$.}
\end{equation}

A unitary matrix $U$ that
transforms $M$ to an upper triangular matrix $A=U^{-1}MU$
of the form (\ref{rkr}) can be
constructed as follows: first find a nonsingular matrix
$S$ such that $J=S^{-1}MS$ is the
Jordan form of $M$ that has diagonal entries in
the prescribed order, then apply the
Gram-Schmidt orthogonalization to the
columns of $S$ and obtain a
unitary matrix $U=ST$, in which $T$ is
upper triangular. Alternatively, a unitary $U$ with the desired
property can be constructed directly, without first obtaining the Jordan
form \cite[Theorem 2.3.1]{HJ1}.

The unitary similarity class of $M$ can contain more than one
upper triangular matrix $A$ of the form (\ref{rkr}).
For example, the argument of any nonzero entry in the first
superdiagonal may be chosen arbitrarily.
The following diagonal unitary similarity permits us to
standardize the choice of these arguments by replacing every nonzero entry
$a_{i,i+1}$ in the first superdiagonal by the nonnegative real
number $r_i:=|a_{i,i+1}|$:
\[
A\mapsto UAU^{-1},\qquad U:=\diag(1,\, u_{1},\,
u_{1} u_{2},\,u_{1} u_{2}u_3,\,\dots),
\]
in which $u_i:=a_{i,i+1}/r_i$ if
$a_{i,i+1}\ne 0$ and $u_i:=1$ if
$a_{i,i+1}=0$. This unitary similarity
is used in the following example.

\begin{example}\label{ob1}
Every square matrix $M$ that is
unitarily similar to a matrix of the
form \eqref{rkr}, in which all entries
of the first superdiagonal of $A$ are
nonzero, is unitarily similar to a
matrix of the form
\begin{equation}\label{rj}
B=\begin{bmatrix}
     \lambda _1&b_{12}&\dots&b_{1n} \\
     &\lambda _2&\ddots&\vdots \\
     &&\ddots&b_{n-1,n}\\ 0&&&\lambda _n
   \end{bmatrix},\quad\begin{matrix}
      \text{all }b_{i,i+1}>0,  \\
               \lambda_1\preccurlyeq
\lambda_2\preccurlyeq\dots\preccurlyeq
\lambda_n.
\end{matrix}
\end{equation}
Such a matrix can be used as a
\emph{canonical form for $M$ under
unitary similarity} since if two
matrices of the form \eqref{rj} are
unitarily similar, then they are
identical. This canonical form is a
special case of a canonical form for
nonderogatory matrices that we
construct in Section \ref{s4}.
\end{example}

The number of Jordan blocks with
eigenvalue $\lambda $ in the Jordan
form of an $n\times n$ matrix $A$ is
equal to $n-\rank (A-\lambda I_n)$.
Thus, a matrix of the form \eqref{rkr} is nonderogatory
if and only if
$ \lambda _i=\lambda_{i+1}$ implies that $a_{i,i+1}\ne 0$.
We formalize this observation in the following lemma.

\begin{lemma}\label{obser}
A matrix is nonderogatory if and only
if it is unitarily similar to a block
matrix of the form
\begin{equation}\label{krao}
A=\begin{bmatrix}
     \Lambda _1 &A_{12}&\dots&A_{1t} \\
     &\Lambda _2&\ddots&\vdots \\
     &&\ddots&A_{t-1,t}\\ 0&&&\Lambda _t
   \end{bmatrix},
   \end{equation}
in which each diagonal block $\Lambda_i$ is $m_i \times m_i$ and has the
form
\begin{equation}\label{rio}
\Lambda _i=\begin{bmatrix}
     \lambda _i &*&\dots&* \\
     &\lambda _i&\ddots&\vdots \\
     &&\ddots&*\\ 0&&&\lambda _i
   \end{bmatrix},\quad \begin{matrix}
   \text{all entries of the first
   superdiagonal}\\ \text{of $\Lambda _i$
   are positive real numbers,}
\end{matrix}
\end{equation}
and the diagonal entries are
lexicographically ordered:
$\lambda_1\prec
\lambda_2\prec\dots\prec \lambda_t.$
\end{lemma}

\section{An algorithm for nonderogatory matrices}
\label{s3}

Let $M$ be a nonderogatory matrix. We
first reduce it by
unitary similarity transformations to a matrix $A$ of
the form described in Lemma
\ref{obser}. Then we reduce $A$ by
transformations $ A\mapsto
A':=U^{-1}AU$ ($U$ is unitary) that
preserve this form.

We prove in Lemma \ref{kjg} that
$A'$ has the form described in Lemma
\ref{obser} if and only if
\[
 U=u_1I_{m_1}\oplus\dots\oplus u_tI_{m_t},\qquad
 |u_1|=\dots=|u_t|=1.
\]
Thus, we reduce $A$ to
canonical form  by transformations
\begin{equation}\label{yag}
A\mapsto\begin{bmatrix}
     \Lambda _1 & u_1^{-1}u_2A_{12}&\dots&
     u_1^{-1}u_tA_{1t} \\
     &\Lambda _2&\ddots&\vdots \\
     &&\ddots&u_{t-1}^{-1}u_tA_{t-1,t}\\ 0&&&\Lambda _t
   \end{bmatrix}.
   \end{equation}
Notice that the blocks $A_{ij}$ are
multiplied by complex numbers of
modulus $1$.

We construct a set of canonical nonderogatory
matrices that includes the
canonical matrices from Example
\ref{ob1}. For this purpose, if
$A_{12}\ne 0$, then we reduce it to the
following form.

\begin{lemma}\label{obss}
Let $C=[c_{ij}]$ be a nonzero $p\times
q$ matrix. Let $c$ be the first nonzero
entry in the sequence formed by the
diagonals of $C$ starting from the
lower left:
\[
c_{p1};\ c_{p-1,1},\ c_{p2};\ c_{p-2,1},\ c_{p-1,2}
,\ c_{p3};\ \dots;\ c_{1q}.
\]
We can replace
$c$ by the positive real number
$r=|c|$ by multiplying $C$ by a complex
number of modulus $1$ . The  resulting matrix is
canonical with respect to
multiplication by complex numbers of modulus
$1$.
\end{lemma}
For example, if the first nonzero
diagonal (starting from the lower left)
of $C$ is below the main diagonal, then
its canonical matrix from Lemma
\ref{obss} has the form
\begin{equation*}\label{grko}
\begin{bmatrix}
\phantom{\ddots}&\phantom{\ddots}&\phantom{\ddots}&
\phantom{\ddots}&\phantom{\ddots}&\phantom{\ddots}&
\phantom{\ddots}&\phantom{\ddots}&*
       \\
*&&&&&&&\phantom{\ddots}
       \\[-4pt]
0&\ddots
       \\[-4pt]
0&\ddots&*
       \\[-6pt]
&\ddots& 0&*
      \\[-2pt]
&&0&r&*
      \\
&&&0&*&*
       \\[-2pt]
&&&&\ddots&\ddots&\ddots
       \\[-2pt]
0&&&&&0&*&*
\end{bmatrix}
       \quad
\begin{matrix}
r\in\mathbb R,\ r>0,  \\
*\text{'s are complex numbers}.\\
\end{matrix}
\end{equation*}

We sequentially reduce the blocks
$A_{ij}$ of the matrix \eqref{krao} to
canonical form in the following order
(i.e., arranging them along the block
superdiagonals of $A$):
\begin{equation}\label{lhu}
A_{12},\ A_{23},\ \dots,\ A_{t-1,t};\ A_{13},\ A_{23},
\ \dots,
\ A_{t-2,t};\ \dots;\ A_{1t}.
\end{equation}

We begin with the block $A_{12}$. If
$A_{12}=0$, then it is not changed by
transformations of the form \eqref{yag}, and so it
is already canonical. If $A_{12}\ne 0$,
then we reduce it as in Lemma
\ref{obss}; to preserve the block
$A_{12}$ obtained, we must impose the condition
$u_1=u_2$ on the transformations
\eqref{yag}.

Then we reduce $A_{23}$ in the same way
and so on, until all blocks in
the first superdiagonal have been reduced.
We obtain a
matrix $A$ in which all nonzero blocks
in the first superdiagonal have the
form described in Lemma \ref{obss}.
This matrix is uniquely determined by
the unitary similarity class of
$A$, up to transformations of the form \eqref{yag}
that satisfy the conditions $u_i=u_{i+1}$
if $A_{i,i+1}\ne 0$;
we say that such transformations are \emph{admissible}.
It is convenient to describe these conditions by a graph
$G^{(1)}$ with vertices $1,\dots,t$ and
with edges $i\text{ --- }(i+1)$ that
correspond to all $A_{i,i+1}\ne 0$.

Next we reduce the
blocks of the second superdiagonal to canonical form.
If $A_{13}=0$ or if
$u_1=u_2=u_3$ (i.e., if $G^{(1)}$ contains the
path $1\text{ --- }2\text{ --- }3$),
then $A_{13}$ is not changed by admissible
transformations of the form \eqref{yag}; it
is already canonical. If $A_{13}\ne 0$
and $G^{(1)}$ does not contain the path $1\text{
--- }2\text{ --- }3$, then we reduce
$A_{13}$ as in Lemma \ref{obss} and add
the edge $1\text{
--- }3$ to the graph. Then we reduce
$A_{24}$ and so on until we have reduced all
blocks in the sequence \eqref{lhu}.

This algorithm can be formalized as
follows. For each graph $G$ with
vertices $1,\dots,t$, we say that
\eqref{yag} is a
\emph{$G$-transformation} if $u_i=u_j$
for all edges $i\text{ --- }j$ in $G$.

\begin{algorithm}\label{aall}
\begin{description}
  \item[]  Let $M$ be a
      nonderogatory matrix, let $A$
      be its upper triangular form
      \eqref{krao} for unitary
      similarity described in Lemma
      \ref{obser}, and let $G_0$ be
      the graph with vertices
      $1,\dots,t$ and without
      edges.

  \item[\it The first step:] We
      construct a pair $(A_1,G_1)$
      as follows. Let $A_{p_1q_1}$
      be the first nonzero block of
      $A$ in the sequence
      \eqref{lhu}. Reduce
      $A_{p_1q_1}$ as in Lemma
      \ref{obss} by transformations of the form
      \eqref{yag} and denote the resulting matrix by
      $A_1$.
      Add the edge $p_1\text{
      --- }q_1$ to $G_0$ and denote
      the resulting graph by $G_1$.

  \item[\it The $\alpha $th step
      $(\alpha \ge 2)$:] Using the
      pair $(A_{\alpha-1
      },G_{\alpha-1 })$ constructed
      at the $(\alpha-1)$st step,
      we construct
      $(A_{\alpha},G_{\alpha })$.
      Let $A_{p_{\alpha }q_{\alpha
      }}$ be the first block of
      $A_{\alpha-1 }$ that is to the right
      of $A_{p_{{\alpha
      }-1}q_{{\alpha }-1}}$ in
      \eqref{lhu} and is changed by
      $G_{{\alpha
      }-1}$-transformations (this
      means that $A_{p_{\alpha
      }q_{\alpha }}\ne 0$ and
      $G_{{\alpha }-1}$ does not contain a path
      from $p_{\alpha }$ to
      $q_{\alpha }$). We reduce
      $A_{p_{\alpha }q_{\alpha }}$
      as in Lemma \ref{obss} and
      denote the resulting matrix by
      $A_{\alpha }$. Add the edge
      $p_{\alpha }\text{
      --- }q_{\alpha }$ to
      $G_{{\alpha }-1}$ and denote
the resulting graph by  $G_{\alpha}$.

\item[\it The result:] The process
    stops at a pair $(A_r,G_r)$
    such that all blocks of $A_r$
    to the right of $A_{p_rq_r}$
    in \eqref{lhu} are not changed
    by $G_r$-transformations. The
    number $r$ of steps is less
    than $t$ since the graph $G_r$
    has $t$ vertices, $r$ edges,
    and no cycles.  Write
$M_{\text{can}}:=A_r$ and $G:=G_r$.
\end{description}
\end{algorithm}

In the proof of Theorem \ref{kie} we
show that the pair $(M_{\text{can}},G)$
is uniquely determined by the
unitary similarity class of $M$; that is,
\emph{$M_{\text{can}}$ is a canonical
form for $M$ with respect to unitary
similarity.}

\section{Canonical nonderogatory
matrices and the classification
theorem}\label{s4}

Algorithm \ref{aall} constructs a
pair $(M_{\text{can}},G)$ for each
nonderogatory matrix $M$. The structure
of $M_{\text{can}}$ is determined by
the graph $G$ as follows:
\begin{itemize}
  \item The blocks
\begin{equation}\label{ctw}
A_{p_1q_1},\ A_{p_2q_2},\ \dots,\ A_{p_rq_r}
\end{equation}
of $M_{\text{can}}$ have the form
described in Lemma \ref{obss}; they
correspond to the edges of $G$.

  \item Let $A_{ij}$ $(i<j)$ be a
      block of $M_{\text{can}}$
      that is not a member of the list in \eqref{ctw}.
      Let
      $A_{p_{\alpha}q_{\alpha}}$ be
      the nearest block in the list
      \eqref{ctw} that is to the left of
      $A_{ij}$ in \eqref{lhu}. If
      there is no such block (i.e., if
      $A_{ij}$ is to the left of
      $A_{p_1q_1}$), we put
      $\alpha: =0$.  Then
\begin{itemize}
  \item[(i)] $A_{ij}=0$ if
      $G_{\alpha }$ does not contain a path
      from $i$ to $j$, and
  \item[(ii)] $A_{ij}$ is
      arbitrary if $G_{\alpha
      }$ contains a path from $i$ to
      $j$.
\end{itemize}
\end{itemize}
The graph $G_{\alpha }$ in (i) and (ii)
can be obtained from $G$ by removing
the edges $u\text{
--- }v$ that correspond to those
$A_{uv}$ in the list \eqref{ctw} that are
reduced after
$A_{p_{\alpha}q_{\alpha}}$ if
$\alpha\ne 0$, and by removing all the
edges of $G$ if $\alpha=0$. Thus,
$A_{uv}$ is to the right of $A_{ij}$ in
\eqref{lhu}; i.e., either $v-u>j-i$, or
$v-u=j-i$ and $u>i$. Hence,
\emph{$M_{\text{\rm can}}$ is a
$G$-canonical matrix} in the sense of
the following definition.

\begin{definition}\label{ddd}
Let $G$ be an undirected graph with vertices
$1,2,\dots,t$ and no cycles. By a
\emph{$G$-canonical matrix}, we mean
a block matrix of the form
\eqref{krao} in which every diagonal
block has the form \eqref{rio},
$\lambda_1\prec
\lambda_2\prec\dots\prec \lambda_t,$
and each block $A_{ij}$ ($i<j$)
satisfies the following conditions:
\begin{itemize}
  \item $A_{ij}$ has the form
      described in Lemma \ref{obss}
      if $G$ contains
      the edge $i\text{
      --- }j$;

  \item $A_{ij}=0$ if either $G$
      contains no path from $i$ to $j$,
      or the path from $i$ to $j$
  (which is unique since $G$
     without cycles)
      contains an edge $u\text{
      --- }v$ ($u<v$) such that
\begin{itemize}
  \item either $v-u>j-i$,
  \item or $v-u=j-i$ and $u>i$;
\end{itemize}

  \item $A_{ij}$ is arbitrary,
      otherwise.
\end{itemize}
\end{definition}

\begin{example}
\begin{itemize}
  \item[(a)] Each matrix of the
      form \eqref{rj} is
      $G$-canonical with
\[
G:\qquad
\xymatrix@C=8mm@R=8mm{
{1}
\ar@{-}[r]&{2}
\ar@{-}[r]&{\cdots}\ar@{-}[r]&
{t}}
\]

  \item[(b)] Each $G$-canonical
      matrix with
\[
G:\qquad\begin{aligned}
\xymatrix@C=8mm@R=8mm{
&{5}\ar@{-}[d]\\
{2}
\ar@{-}[r]&{3}
\ar@{-}[r]&{4}\ar@{-}[r]&
{1}}
\end{aligned}
\]
has the form
\[\begin{bmatrix}
     \Lambda _1 &0&0&C_4&\convolution    \\
     &\Lambda _2&C_1&\convolution   &\convolution  \\
     &&\Lambda _3&C_2&C_3\\ &&&\Lambda _4&0\\
     0&&&&\Lambda _5
   \end{bmatrix}\]
in which
\begin{itemize}
  \item each block $\Lambda _i$
      has the form
      \eqref{rio} and
      $\lambda_1\prec
      \lambda_2\prec\lambda_3\prec\lambda_4\prec
      \lambda_5$,

  \item each block $C_i$ has
      the form described in
      Lemma
\ref{obss},
  \item the stars denote
      arbitrary blocks.
\end{itemize}
\end{itemize}

\end{example}

A $G$-canonical matrix is a
\emph{canonical nonderogatory block} if
$G$ is a tree. It follows from the
uniqueness in (b) of the next theorem
that canonical nonderogatory blocks are
indecomposable under unitary similarity,
i.e., they are not unitarily similar
to a direct sum of square matrices of
smaller sizes. Their role is
analogous to the role of Jordan blocks
in the Jordan canonical form.

\begin{theorem}\label{kie}
{\rm(a)} For each nonderogatory matrix
$M$, there is a unique undirected graph $G$ and
a unique $G$-canonical matrix
$M_{\text{\rm can}}$ such that $M$ is
unitarily similar to $M_{\text{\rm
can}}$. Thus, $M$ is unitarily
similar to $N$ if and only if $M_{\text{\rm
can}}=N_{\text{\rm can}}$.

{\rm(b)} Each nonderogatory matrix $M$
is unitarily similar to a direct sum of
canonical nonderogatory blocks. This
direct sum is uniquely determined by
$M$, up to permutation of summands.
\end{theorem}

\begin{remark}\label{rrd}
The direct sum in Theorem \ref{kie}(b)
is permutationally similar to
$M_{\text{\rm can}}$ and can be
obtained from it as follows: The graph $G$ is a
disjoint union of trees; denote them by
$G_1,\dots,G_s$. Let
$u_{i1}<u_{i2}<\dots<u_{it_i}$ be the
vertices of $G_i$. Let $A_i$ be the
$t_i\times t_i$ submatrix of
$M_{\text{\rm can}}$ formed by rows
$u_{i1},\dots,u_{it_i}$ and columns
$u_{i1},\dots,u_{it_i}$. Definition
\ref{ddd} ensures that the $u_{il},\!u_{jk}$ block
of $M_{\text{\rm can}}$ is zero if
$i\ne j$. Therefore, $M_{\text{\rm
can}}$ is permutationally similar to
\[
A_1\oplus A_2\oplus\dots\oplus A_s,
\]
which is the desired direct sum. Each
$A_i$ is a $G_i'$-canonical matrix,
in which $G_i'$ is the tree obtained from
$G_i$ by relabeling the vertices
$u_{i1},\dots,u_{it_i}$ with
$1,\dots,t_i$.
\end{remark}

\section{Proof of Theorem
\ref{kie}}\label{s5}

Our proof is based on the following
lemma about unitary similarity of
matrices of the form described in Lemma \ref{obser}.
This lemma was
proved in greater generality in \cite{lit}
and in \cite{mit}; we
offer a proof for the reader's
convenience.

\begin{lemma}\label{kjg}
Let
\begin{equation*}\label{kju}
A=\begin{bmatrix}
     \Lambda _1 &A_{12}&\dots&A_{1t} \\
     &\Lambda _2&\ddots&\vdots \\
     &&\ddots&A_{t-1,t}\\ 0&&&\Lambda _t
   \end{bmatrix},\qquad
B=\begin{bmatrix}
     \Lambda' _1 &B_{12}&\dots&B_{1t'} \\
     &\Lambda' _2&\ddots&\vdots \\
     &&\ddots&B_{t'-1,t'}\\ 0&&&\Lambda' _{t'}
   \end{bmatrix}
   \end{equation*}
be nonderogatory matrices of the
form described in Lemma \ref{obser}.
Assume that they are unitarily similar:
$U^{-1}AU=B$ with unitary $U$. Then
$t=t'$,
\begin{equation}\label{pos}
\Lambda_1=\Lambda'_1,\ \dots,\
 \Lambda_t=\Lambda'_t,
\end{equation}
and $U$ has the form
\begin{equation}\label{kpw}
 U=u_1I_{m_1}\oplus\dots\oplus u_tI_{m_t}
\end{equation}
in which $u_1,\dots,u_t$ are complex
numbers of modulus $1$ and the size of $\Lambda_i$
is $m_i\times m_i$ for each $i$.
\end{lemma}

\begin{proof} The matrices $A$ and $B$
have the same main diagonal since they
are similar and the entries along their
main diagonals are lexicographically
ordered.  This means that $t=t'$ and
for each $i$ the diagonal blocks
$\Lambda _i$ and $\Lambda'_i$ are
$m_i\times m_i$ matrices of the form
\eqref{rio} with the same $\lambda _i$.
The proof is divided into three steps.

\emph{Step 1: Prove that $U$ has the
form
\begin{equation}\label{dwo}
U=U_1\oplus U_2\oplus\dots\oplus U_t
\end{equation}
in which every block $U_i$ is
$m_i\times m_i$.} If $t =1$ there is
nothing to prove, so assume that $t\ge
2$. Partition $U$ into blocks $U_{ij}$
of size $m_i\times m_j$. Our strategy
is to exploit the equality of
corresponding blocks of both sides of
the identity $AU = UB$.

The $t,\!1$ block of $AU$ is
$\Lambda_tU_{t1}$, and the $t,\!1$
block of $UB$ is $U_{t1}\Lambda'_1$.
Since $\lambda_t\ne\lambda _1$,
$U_{t1}=0$ is the only solution to
$\Lambda_tU_{t1}=U_{t1}\Lambda'_1$.

If $t> 2$, then the $t,\!2$ block of
$AU$ is $\Lambda_{t} U_{t2}$, and the
$t,\!2$ block of $UB$ is
$U_{t2}\Lambda'_{2}$ (since $U_{t1}
=0$); we have $\Lambda_{t} U_{t2}=
U_{t2}\Lambda'_{2}$. Since
$\lambda_{t}\ne\lambda _{2}$, we have
$U_{t2} =0$. Proceeding in this way
across the last block row of $AU = UB$,
we find that
$U_{t1},U_{t2}\dots,U_{t,t-1}$ are all
zero.

Now equate the blocks of $AU = UB$ in
positions $(t-1),\!k$ for $k
=1,2,\dots,t-2$ and conclude in the
same way that
$U_{t-1,1},U_{t-1,2},\dots,U_{t-1,t-2}$
are all zero. Working our way up the
block rows of $AU = UB$, left to right,
we conclude that $U_{ij} =0$ for all
$i>j$. Since
$U^{-1}=U^*$, it follows that $U_{ij} =0$ for all
$j>i$ and hence
\[U =
U_{11}\oplus U_{22}\oplus \dots\oplus
U_{tt}.\]
This proves \eqref{dwo} with
$U_i:=U_{ii}$.

\emph{Step 2: Prove that $U$ is
diagonal.} Since $AU = UB$, we have $t$
identities $\Lambda _{i}U_i =
U_i\Lambda'_{i},$ $i =1,\dots,t$, and
all the entries in
the first superdiagonal of each
$\Lambda_{i}$ and $\Lambda'_{i}$ are
positive real numbers. Thus, it
suffices to consider the case $t =1$.
In this case
\[
A=\begin{bmatrix}
     \lambda &a_{12}&\dots&a_{1n} \\
     &\lambda&\ddots&\vdots \\
     &&\ddots&a_{n-1,n}\\ 0&&&\lambda
   \end{bmatrix},\qquad
B=\begin{bmatrix}
\lambda &b_{12}&\dots&b_{1n} \\
     &\lambda&\ddots&\vdots \\
     &&\ddots&b_{n-1,n}\\ 0&&&\lambda
   \end{bmatrix},
\]
$a_{i,i+1}$ and $b_{i,i+1}$ are
positive real numbers for all $i
=1,\dots,n-1$, and $AU = UB$. As in
Step 1, we equate corresponding entries
of the identity
\begin{equation}\label{gep1}
(A-\lambda I_n)U =
U(B-\lambda I_n).
\end{equation}
In position $n,\! 1$ we have $0=0$. In
position $n,\! 2$ we have
$0=u_{n1}b_{12}$; since $b_{12}\ne 0$
it follows that $u_{n1} =0$. Proceeding
across the last row of \eqref{gep1}, we
obtain
\[u_{n1}=u_{n2}=\dots=u_{n,n-1}=0.\]
Working our way up the rows of
\eqref{gep1} in this fashion, left to
right, we find that $u_{ij} =0$ for all
$i>j$. Thus, $U$ is upper triangular.
Since $U$ is unitary, it is diagonal:
$U=\diag(u_1,\dots,u_n)$.

\emph{Step 3: Prove that $U=\diag(u_1,\dots,u_n)$ has the
form \eqref{kpw}.} We continue to
assume that $t=1$. Equating the entries
of $AU = UB$ in position $i,\! i +1$,
we have $a_{i,i+1}u_{i+1} =
u_{i}b_{i,i+1}$, so
$a_{i,i+1}/b_{i,i+1} = u_{i}/u_{i+1}$,
which is positive real and has modulus
one. We conclude that $u_{i}/u_{i+1}
=1$ for each $i =1,\dots,n-1$, and
hence $u_1=\dots = u_n$. This proves
\eqref{kpw}, which implies \eqref{pos}.
\end{proof}

\begin{proof}[Proof of Theorem
\ref{kie}] (a) Let $M$ be a
nonderogatory matrix. Algorithm
\ref{aall} constructs the graph $G$ and
the matrix $M_{\text{can}}$, which is
unitarily similar to $M$. As shown at
the beginning of Section \ref{s4},
$M_{\text{can}}$ is a $G$-canonical
matrix.

Let $M$ and $N$ be nonderogatory
matrices that are unitarily similar.
Our goal is to prove that Algorithm
\ref{aall} reduces them to the same
matrix $M_{\text{can}}=N_{\text{can}}$.
Following the algorithm, we first
reduce $M$ and $N$ to matrices $A$ and
$B$ of the form described in Lemma
\ref{obser}. They are unitarily
similar; that is, $U^{-1}AU=B$ for a
unitary matrix $U$. Lemma \ref{kjg} ensures that
$t=t'$, $\Lambda_i=\Lambda'_i$ for all
$i$, and $U$ has the form \eqref{kpw}.
This means that $B$ is obtained from
$A$ by a transformation of the form \eqref{yag}:
\begin{equation}\label{wig}
B_{ij}=u^{-1}_iu_jA_{ij},\qquad
 |u_i|=1,\qquad i,j=1,\dots,t.
\end{equation}

We arrange the superdiagonal blocks
$A_{ij}$ in $A$ and $B_{ij}$ in $B$
along the block superdiagonals, as in
\eqref{lhu}. By \eqref{wig}, the first
nonzero superdiagonal block of $A$ and
the first nonzero superdiagonal block
of $B$ occur in the same
position $p_1,\!q_1$. In Step 1 of
Algorithm \ref{aall}, we reduce them to
the same form described in Lemma
\ref{obss} and obtain the matrices
$A_1$ and $B_1$, in which the
$p_1,\!q_1$ blocks are equal.

In Step $\alpha$, we reduce the first
superdiagonal block of $A_{\alpha-1}$
that is changed by
$G_{\alpha-1}$-transformations, and the
first superdiagonal block of
$B_{\alpha-1}$ that is changed by
$G_{\alpha-1}$-transformations. They
occur in the same position
$p_{\alpha},\!q_{\alpha}$ and are
reduced to the same form described in
Lemma \ref{obss}. We obtain the
matrices $A_{\alpha}$ and $B_{\alpha}$,
in which the blocks in position
$p_{\alpha},\!q_{\alpha}$ coincide;
the superdiagonal blocks that precede them
coincide as well. The matrix $B_{\alpha}$
can be obtained from $A_{\alpha}$ by a
$G_{\alpha}$-transformation, which
preserves these blocks.

The process stops at a matrix $A_r$
such that none of its blocks are
changed by $G_r$-transformations. Then
$A_r=B_r$ and so
$M_{\text{can}}=N_{\text{can}}$.

(b) This statement follows from Remark
\ref{rrd}.
\end{proof}

\section{Canonical matrix pairs for similarity}
\label{s6}

Let $(M,N)$ be a pair of $n\times n$
matrices, and let $M$ have $n$ distinct
eigenvalues. In this section, we give a
canonical form for $(M,N)$ with respect
to the similarity transformations
\eqref{hep}.

The pair $(M,N)$ is similar to some
pair
\begin{equation}\label{rdt}
(A,B)=\left(\begin{bmatrix}
     \lambda _1 &&0 \\
     &\ddots& \\
     0&&\lambda _n
   \end{bmatrix}, \begin{bmatrix}
     b_{11} &\dots&b_{1n} \\
     \vdots&\ddots&\vdots \\
     b_{n1}&\dots&b_{nn}
   \end{bmatrix}\right),\quad \lambda_1\prec\dots\prec
\lambda_n,
\end{equation}
in which $\prec$ is the strict
lexicographic order on $\mathbb C$; see
\eqref{gep}.

Let \[ (A',B')=\left(\begin{bmatrix}
     \lambda' _1 &&0 \\
     &\ddots& \\
     0&&\lambda' _n
   \end{bmatrix}, \begin{bmatrix}
     b_{11}' &\dots&b'_{1n} \\
     \vdots&\ddots&\vdots \\
     b_{n1}'&\dots&b'_{nn}
   \end{bmatrix}\right),
   \quad \lambda'_1\prec\dots\prec
\lambda'_n,
\]
be another pair of this form, and let
it be similar to $(A,B)$; that is,
$(S^{-1}AS,S^{-1}BS)=(A',B')$ for some
nonsingular $S$. Then $A=A'$, $AS=SA$,
and so $S=\diag(s_1,\dots,s_n)$ in
which $s_1,\dots,s_n\in\mathbb C$.
Thus, the pair \eqref{rdt} is uniquely
determined by $(M,N)$, up to
transformations
\begin{equation}\label{fwo}
B\mapsto B'=\begin{bmatrix}
b_{11} &s_1^{-1}s_2b_{12}&\dots&s_1^{-1}s_nb_{1n} \\
s_2^{-1}s_1b_{21}&b_{22}&\dots&s_2^{-1}s_nb_{2n} \\
\vdots&\vdots&\ddots&\vdots \\
s_n^{-1}s_1b_{n1}&s_n^{-1}s_2b_{n2}&\dots&b_{nn}
   \end{bmatrix}
\end{equation}
in which $s_1,\dots,s_n$ are arbitrary
nonzero complex numbers.

\begin{example}
Suppose that a pair $(M, N)$ of
$n\times n$ matrices is similar to a
pair of the form \eqref{rdt} in which
$b_{12},b_{13},\dots,b_{1n}$ are all
nonzero. Taking $s_1=1$,
$s_2=b_{12}^{-1}$, \dots,
$s_n=b_{1n}^{-1}$ in \eqref{fwo}, we
reduce $(M,N)$ to the form
\begin{equation}\label{jpf}
\left(\begin{bmatrix}
     \lambda_1&&&0 \\
     &\lambda _2 \\
     &&\ddots\\
     0&&&\lambda _n\\
   \end{bmatrix},
\begin{bmatrix}
     *&1&\dots&1 \\
     *&*&\dots&* \\
     \vdots&\vdots&\ddots&\vdots\\
     *&*&\dots&*
   \end{bmatrix}\right),\quad \lambda_1\prec\dots\prec
\lambda_n,
   \end{equation}
in which the stars denote arbitrary
complex numbers. We can use \eqref{jpf}
as a \emph{canonical form for $(M,N)$
for similarity} since if $B$ and $B'$
in \eqref{fwo} have $1$ in positions
$1,\!k$, $k=2,3, \dots,n$, then
$s_1=\dots=s_n$, and so $B=B'$. Thus,
if pairs of the form \eqref{jpf} are
similar, then they are equal.
\end{example}

In the general case, we reduce  $B$ by
transformations of the form \eqref{fwo} using the
following algorithm. We arrange the
entries of $B$ along the rows starting
from the first; that is, $b_{ij}$
\emph{precedes} $b_{pq}$ if $(i,j)\prec
(p,q)$ with respect to the
lexicographic order. For each directed
graph $G$ with vertices $1,\dots,n$, we
say that \eqref{fwo} is a
\emph{$G$-transformation} if $s_i=s_j$
for all directed edges $i\to j$ in $G$.

\begin{algorithm}\label{alj}
\begin{description} \item[] Let
$B=[b_{ij}]$ be an $n\times n$
      matrix. Denote by $G_0$ the
      graph with vertices
      $1,\dots,n$ and without
      edges.

\item[\it The first step.] The
    entry $b_{11}$ is not changed
    by transformations of the form \eqref{fwo};
    we mark it as reduced and
    write $(B_1,G_1):=(B,G_0)$.

\item[\it The second step.] If
    $b_{12}=0$ then it is not
    changed by
    $G_1$-transfor\-mations, we
    mark $b_{12}$ as reduced
    and write
    $(B_2,G_2):=(B_1,G_1)$. If
    $b_{12}\ne 0$ then we make
    $b_{12}=1$ by
    $G_1$-transformations, add the
    directed edge $1\to 2$ to $G_1$, and
    denote by $B_2$ and $G_2$ the resulting
   matrix and directed
    graph.

\item[\it The $k$th step.] Let
    $b_{pq}$ be the $k$th entry;
    that is, $(p-1)n+q=k$. If
    $p=q$, or $b_{pq}=0$, or
    $G_{k-1}$ has an undirected
    path from $p$ to $q$, then
    $b_{pq}$ is not changed by
    $G_{k-1}$-transformations; we
    mark $b_{pq}$ as reduced
    and write
    $(B_k,G_k):=(B_{k-1},G_{k-1})$.
    Otherwise, we make $b_{pq}=1$
    by $G_{k-1}$-transformations,
    add the directed edge $p\to q$ to
    $G_{k-1}$, and denote by $B_k$
    and $G_k$ the resulting matrix
    and directed graph.

\item[\it The result.] After $n^2$
    steps, we obtain a matrix
    $B_{n^2}$, in which all entries
    have been marked as reduced. Write
    $(B_{\text{can}},G):=(B_{n^2},G_{n^2})$.
\end{description}
\end{algorithm}

Let us show that \emph{$B_{\text{can}}$
is a canonical form for $B$ with respect
to transformations of the form \eqref{fwo}}; that
is, if $B$ and $C$ are $n\times n$
matrices such that $B$ can be reduced
to $C$ by transformations of the form \eqref{fwo}
then $B_{\text{can}}=C_{\text{can}}$.
Indeed, after $k$ steps of Algorithm
\ref{alj} applied to $B$ and $C$, we
obtain the matrices $B_k$ and $C_k$ and
the same directed graph $G_k$. One can prove by
induction on $k$ that $B_k$ reduces to
$C_k$ by $G_k$-transformations, and so
the first $k$ entries of $B_k$ and
$C_k$ coincide. Taking $k=n^2$, we
obtain $B_{\text{can}}=C_{\text{can}}$.

Let $(M,N)$ be a pair of $n\times n$
matrices, and let $M$ have $n$ distinct
eigenvalues. Then $(M,N)$ is similar to
a pair $(A,B)$ of the form \eqref{rdt},
which is uniquely determined by
$(M,N)$, up to transformations of the form
\eqref{fwo}. Taking
$B=B_{\text{can}}$, we obtain the pair
$(M,N)_{\text{can}}:=(A,B_{\text{can}})$,
which is similar to $(M,N)$ and is
uniquely determined by $(M,N)$. Thus,
\begin{equation}\label{kux}
\text{\it $(M,N)_{\text{\rm can}}$ is a
canonical form for $(M,N)$ for similarity.}
\end{equation}

In the $k$th step of Algorithm
\ref{alj}, we reduce the $k$th entry
$b_{pq}$ and construct the directed graph
$G_{k}$. The graph $G_{k}$ can be also
obtained from $G=G_{n^2}$ by removing
the directed edges $i\to j$ that correspond to
those entries $b_{ij}$ that were
reduced to $1$ after $b_{pq}$; this
means that $(i,j)\succ (p,q)$. Thus,
\begin{equation}\label{ftj}
\text{\it the pair $(M,N)_{\text{can}}$ is
$G$-canonical}
\end{equation}
in the sense of the following
definition.

\begin{definition}\label{jjt}
Let $G$ be a directed graph with
vertices $1,2,\dots,n$ and no
undirected cycles. By a
\emph{$G$-canonical matrix pair} we
mean a matrix pair of the form
\begin{equation*}\label{rjt}
\left(\begin{bmatrix}
     \lambda _1 &&0 \\
     &\ddots& \\
     0&&\lambda _r
   \end{bmatrix}, \begin{bmatrix}
     b_{11} &\dots&b_{1r} \\
     \vdots&\ddots&\vdots \\
     b_{r1}&\dots&b_{rr}
   \end{bmatrix}\right),\qquad \lambda_1\prec\dots\prec
\lambda_r,
\end{equation*}
in which every entry $b_{pq}$ satisfies
the following conditions:
\begin{itemize}
  \item[(i)] $b_{pq}=1$ if $G$ has
      the directed edge $p\to q$;

  \item[(ii)] $b_{pq}=0$ if either
      $G$ has no undirected path
      from $p$ to $q$, or the
      undirected path from $p$ to
      $q$ contains a directed edge $i\to
      j$ such that $(i,j)\succ
      (p,q)$ with respect to the
      lexicographic order;

  \item[(iii)] $b_{pq}$ is
      arbitrary, otherwise.
\end{itemize}
\end{definition}

\begin{example}
\begin{itemize}
  \item[(a)] Each pair of the form
      \eqref{jpf} is $G$-canonical
      with
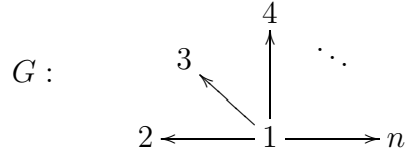
\[
G:\qquad\begin{aligned}
\xymatrix@C=1mm@R=0mm{&&&&{4}\ar@{<-}[dddd]
\\
&{3}\ar@{<-}[dddrrr]&&&&&{\!\!\ddots\!\!}\\\\\\
{2}\ar@{<-}[rrrr]&&&&{1}
\ar@{->}[rrrr]&&&&
{n}}
\end{aligned}
\]

  \item[(b)] Each $G$-canonical
      matrix pair with
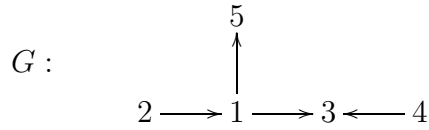
\[
G:\qquad\begin{aligned}
\xymatrix@C=8mm@R=8mm{
&{5}\ar@{<-}[d]\\
{2}
\ar@{->}[r]&{1}
\ar@{->}[r]&{3}\ar@{<-}[r]&
{4}}
\end{aligned}
\]
has the form
\[
\left(\begin{bmatrix}
     \lambda_1&0&0&0&0 \\
     0&\lambda _2&0&0&0 \\
     0&0&\lambda _3&0&0\\
     0&0&0&\lambda _4&0\\
     0&0&0&0&\lambda _5
   \end{bmatrix},
\begin{bmatrix}
     *&0&1&0&1 \\
     1&*&*&0&* \\
     *&*&*&0&*\\
     0&0&1&*&*\\
     *&*&*&*&*
   \end{bmatrix}\right)
\]
in which $\lambda_1\prec\dots\prec
\lambda_5$ and the stars denote
arbitrary complex numbers.
\end{itemize}
\end{example}

A $G$-canonical matrix pair is an
\emph{indecomposable canonical matrix
pair} if $G$ is a tree. It is not
similar to a direct sum of pairs of
square matrices of smaller sizes. This is a consequence
of the uniqueness assertion in (b) of
the following theorem.

\begin{theorem}\label{kid}
{\rm(a)} For each pair $(M,N)$ of
$n\times n$ matrices in which $M$ has
$n$ distinct eigenvalues, there exist a
unique directed graph $G$ and a unique
$G$-canonical matrix pair
$(M,N)_{\text{\rm can}}$ such that
$(M,N)$ is similar to $(M,N)_{\text{\rm
can}}$. Thus, $(M,N)$ is similar to
$(M',N')$ if and only if $(M,N)_{\text{\rm
can}}=(M',N')_{\text{\rm can}}$.

{\rm(b)} Each pair $(M,N)$ of $n\times
n$ matrices in which $M$ has $n$
distinct eigenvalues is similar to a
direct sum of indecomposable canonical
matrix pairs. This direct sum is
uniquely determined by $(M,N)$, up to
permutation of summands.
\end{theorem}

The statement (a) of Theorem \ref{kid}
follows from \eqref{kux} and \eqref{ftj}.
The statement (b) is a consequence of the
following remark.

\begin{remark}\label{ry}
The direct sum in Theorem \ref{kid}(b)
is permutationally similar to
$(M,N)_{\text{\rm can}}$ and can be
obtained from it as follows: The directed graph $G$ is a
disjoint union of trees; denote them by
$G_1,\dots,G_s$. Let
$u_{i1}<u_{i2}<\dots<u_{it_i}$ be the
vertices of $G_i$. Let $(A_i,B_i)$ be
the pair of $t_i\times t_i$ submatrices
of the matrices in $(M,N)_{\text{\rm
can}}$ formed by rows
$u_{i1},\dots,u_{it_i}$ and columns
$u_{i1},\dots,u_{it_i}$.
Definition \ref{jjt} ensures that the $u_{il},\!u_{jk}$
entries of the matrices in
$(M,N)_{\text{\rm can}}$ are zero if
$i\ne j$. Therefore, $(M,N)_{\text{\rm
can}}$ is permutationally similar to
\[
(A_1,B_1)\oplus (A_2,B_2)\oplus\dots\oplus (A_s,B_s),
\]
which is the desired direct sum. Each
$(A_i,B_i)$ is a $G_i'$-canonical
matrix pair, in which $G_i'$ is the tree
obtained from $G_i$ by relabeling its
vertices $u_{i1},\dots,u_{it_i}$ as
$1,\dots,t_i$.
\end{remark}

\end{document}